\documentclass[reqno]{amsart}

\usepackage{amsmath}
\usepackage{amsfonts}
\usepackage{rotating}
\usepackage{bbm}
\usepackage[all,cmtip]{xy}
\usepackage{amscd}
 
\usepackage{tikz}
 \usepackage{pgfplots}

 \usetikzlibrary{arrows}
 \tikzset{>=latex}
\usetikzlibrary{arrows.meta}
\pgfplotsset{compat=1.10}
\usetikzlibrary{decorations.markings}

\usepgfplotslibrary{fillbetween}
\usetikzlibrary{patterns}
\usepackage{mathrsfs}
\usepackage{color}
\usepackage{xcolor}
\usepackage{mathtools}
\usepackage{amssymb}
\usepackage[latin1]{inputenc}
\usepackage{graphicx}
\usepackage{dsfont}
\usepackage{color}
\usepackage{hyperref}
\usepackage{enumerate}
\usepackage{marginnote}

\usepackage{subcaption}
 \usepackage{comment}

\newtheorem{theorem}{Theorem}[section]

\newtheorem{lemma}[theorem]{Lemma}
\newtheorem{proposition}[theorem]{Proposition}

\theoremstyle{definition}
\newtheorem{definition}[theorem]{Definition}

\newtheorem{example}[theorem]{Example}

\theoremstyle{remark}

\newcommand{\C}{{\mathbb{C}}}

\newcommand{\R}{{\mathbb{R}}}

\newcommand{\Z}{{\mathbb{Z}}}

\renewcommand{\epsilon}{\varepsilon}
\renewcommand{\theta}{\vartheta}

\usepackage{graphicx}
 
\begin{document}

\title[]{A variational approach to  time-dependent planar two-center Stark-Zeeman systems}

 \author[U. Frauenfelder]{Urs Frauenfelder}
   \address{Department of Mathematics, University of Augsburg, 86159 Augsburg, Germany}
   \email {urs.frauenfelder@math.uni-augsburg.de}

 \author[S. Kim]{Seongchan Kim}
   \address{ Department of Mathematics Education, Kongju National University, Gongju 32588, Republic of Korea}
   \email {seongchankim@kongju.ac.kr}

\setcounter{tocdepth}{3}

\date{Recent modification; \today}
 \begin{abstract}
 
 We study periodic orbits in a time-dependent two-center Stark-Zeeman system, which models the motion of a charged particle attracted by two fixed Coulomb centers and subject to external magnetic and    time-dependent  electric   fields. A motivating example is provided by the bicircular restricted four-body problem which investigates the motion of a massless particle, influenced by the Newtonian gravitational attraction of the earth, moon and  periodically moving sun. Due to   singularities at the Coulomb centers, standard local variational approaches fail. To overcome this, we employ the Birkhoff regularization map and   construct a non-local regularized  action functional on a blown-up loop space, following a recent method due to Barutello-Ortega-Vernizi \cite{BOV21}.  We show that the    critical points of this regularized action functional satisfy a certain second-order delay differential equation and correspond to periodic solutions of the original system,   including collisions. Additionally, we examine the symmetric structure of the regularized functional. 

  \end{abstract}
   \maketitle

   


\section{Introduction}

A two-center Stark-Zeeman system describes the dynamics of an electron attracted by two proton subject to an external electric and magnetic field. We allow to electric field to depend periodically on time. We are interested in studying periodic solutions of the electron. It might happen that the electron along its periodic orbit collides sometimes with the protons. In the autonomous case where the electric field does not depend on time a variational approach to periodic solutions which might have collisions can be obtained by blowing up the energy hypersurface of the electron. There are various versions how to do that. The Levi-Civita regularization based on the complex squaring map regularizes just collisions with one of the protons. The Birkhoff regularization regularizes collisions with both protons simultaneously.  See \cite{CFvK17SZ, CFZ23SZ} for further details. 

A new regularization technique was recently discovered by Barutello, Ortega, and Verzini \cite{BOV21}. In contrast to the classical regularization, Barutello, Ortega, and Verzini do not blow-up the energy hypersurface, but the loop space. In particular, this new regularization technique is of great interest in the case where the Hamiltonian system is not autonomous, but depends periodically on time, since in this case there is no preserved energy and therefore no energy hypersuface which can be blown-up. Barutello, Ortega, and Verzini also use the complex squaring map, but combine it with a reparametrization of the loop. We refer to this as \emph{Levi-Civita type BOV-regularization}. In \cite{Fra25SZ} it was explained how Levi-Civita type BOV-regularization can be applied to obtain a variational approach to collisional periodic orbits in one-center Stark-Zeeman systems, where there is only one proton.

In this note we show that the new approach can as well be applied to the case where there are two protons and the electron is allowed to collide with both of them. To regularize collisions with both protons simultaneously we develop a \emph{Birkhoff type BOV-regularization} and prove
\\ \\
\textbf{Theorem\,A: } \emph{Collisional periodic orbits in two-center Stark-Zeeman systems can be variationally obtained via Birkhoff type BOV-regularization.}
\\ \\
The mathematically precise statement of Theorem\,A can be found in Theorem~\ref{thm:main}.

\subsection*{Acknowledgments}
Urs Frauenfelder acknowledges partial support by DFG project
FR 2637/6-1.

 \section{Time-dependent two-center Stark-Zeeman systems}\label{sec:SZ}

We follow the framework developed in \cite{CFZ23SZ, Fra25SZ}.

Let $E$ and $M$ be two distinct points in $\R^2 \cong \C,$ representing the positions of two massive bodies, referred to as the earth and the moon, respectively. Let $ 1-\mu $ and $\mu>0$ denote  their respective masses.

Assume that an open subset $  \mathfrak{D}_0 \subset \C$  contains the two points $E$ and $M$ and define $\mathfrak{D}  := \mathfrak{D}_0 \setminus \{ E, M \}.$ Given a smooth function $E \colon S^1 \times \mathfrak{D}_0  \to \R,$ we write 
\[
E_t \colon \mathfrak{D}_0 \to \R, \quad E_t(q) = E(t, q )
\]
for each $t \in S^1.$ We then define the time-dependent potential
\[
V_t \colon \mathfrak{D} \to \R, \quad V_t(q):= E_t(q) - \frac{ 1-\mu }{\lvert q- E\rvert} - \frac{\mu }{\lvert q - M\rvert}
\]
  and the associated Hamiltonian
\[
H_t \colon T^* \mathfrak{D} \to \R, \quad H_t(q,p):=\frac{1}{2}\lvert p \rvert^2 + V_t(q).
\]

Let $\mathfrak{B} \colon \mathfrak{D}_0 \to \R$ be a smooth function, called the magnetic field.  It defines a magnetic  two-form 
\[
\sigma_{\mathfrak{B}} := \mathfrak{B} dq_1 \wedge dq_2 \in \Omega^2 ( \mathfrak{D}_0 ) ,
\]
which  lifts to the twisted symplectic form  
\[
\omega_{\mathfrak{B}} = \sum_{j=1}^2 dp_j \wedge dq_j + \pi^* \sigma_{\mathfrak{B}} \in \Omega^2( T^* \mathfrak{D}_0),
\]
where $\pi \colon T^* \mathfrak{D}_0 \to \mathfrak{D}_0$ denotes the footpoint projection. 
Since the second de Rham cohomology group of $\mathfrak{D}_0$ is trivial, there exists a one-form $A \in \Omega^1(\mathfrak{D}_0)$ such that 
\begin{equation}\label{eq:primitiveA}
\sigma_{\mathfrak{B}} = dA.
\end{equation}

The dynamics of the time-dependent two-center Stark-Zeeman system is governed by the flow of the twisted Hamiltonian vector field $X_{H_t}^{\mathfrak{B}}$  implicitly defined~by
\[
d H_t = \omega_{\mathfrak{B}}( \cdot, X_{H_t}^{\mathfrak{B}}).
\]
A periodic orbit $x \colon S^1 \to T^* \mathfrak{D}$ is a solution to the first-order ODE 
\begin{equation}\label{eq:twistedODE}
\dot{x}(t) =  X_{H_t}^{\mathfrak{B}} (x(t)), \quad t \in S^1.
\end{equation}
Due to the time-dependence of  the Hamiltonian,  energy is not conserved, and   periodic orbits do  not lie on  fixed  energy levels. 
Writing $x(t) = (q(t),p(t)),$ we may recast   the first-order ODE \eqref{eq:twistedODE} for $x$ into a second-order ODE for the position $q$ by
\begin{equation}\label{eq:twistedODEforq}
\begin{aligned}
\ddot{q}  &= -\mathfrak{B}(q) i \dot{q} -    \nabla V_t(q)\\
&= -\mathfrak{B}(q) i \dot{q} - \frac{(1-\mu) (q-E)}{\lvert q- E\rvert^3 } - \frac{\mu(q-M)}{\lvert q - M \rvert^3} - \nabla E_t(q).
\end{aligned}
\end{equation}

Let   $\mathfrak{L}\mathfrak{S}$ denote  the free loop space a  set $\mathfrak{S}.$
Define the   $L^2$-inner product $\left< \cdot, \cdot \right> \colon \mathfrak{L}\C \times \mathfrak{L}\C \to \R$   by
\[
\left< \xi, \eta \right> = \int_0^1 \left< \xi(t), \eta(t) \right> dt = \int_0^1 {\rm Re}(\bar{\xi}(t){\eta(t)})dt, \quad \quad \xi, \eta \in \mathfrak{L}\C,
\]
and let the associated $L^2$-norm be given by
\[
\| \xi \|:=\sqrt{ \left< \xi, \xi \right>}.
\]
 One readily verifies  that any solution $q$ to  ODE \eqref{eq:twistedODEforq} is a critical point of the action functional $\mathscr{A} \colon \mathfrak{L}\mathfrak{D} \to \R$ defined by
\begin{equation*}\label{eq:unregularizedfunctional}
\begin{aligned}
  \mathscr{A}(q)  = \frac{1}{2}\|\dot{q} \|^2  - \int_{S^1} q^*A - \int_0^1 V_t (q(t))dt,
\end{aligned}
\end{equation*}
where $A$ is any  primitive of the magnetic form $\mathfrak{B},$ see \eqref{eq:primitiveA}.

 \begin{definition} 
A continuous map   $q \colon S^1 \to \C$   is called a {\it generalized solution} of \eqref{eq:twistedODEforq} if it satisfies the following conditions:
\begin{enumerate}
    \item the collision set $\mathfrak{C}_q = \{ t \in S^1 \mid q(t) \in \{E, M \}\}$ is finite;
    
    \item  on the complement $S^1 \setminus \mathfrak{C}_q$ the map $q$ is smooth and satisfies \eqref{eq:twistedODEforq};  
    
    \item  the energy 
    \[
     \frac{1}{2}\lvert \dot{q}(t)\rvert ^2 -\frac{ 1-\mu }{\lvert q(t)-E\rvert}-\frac{\mu}{\lvert q(t)-M\rvert} + \int_t^1 \dot{E}_s(q(s))ds + E_t(q(t)),  
    \]
    defined on $ S^1 \setminus \mathfrak{C}_q,$
    extends to a continuous function on the whole circle. 
\end{enumerate}

\end{definition}
\noindent
Note that the energy is constant along every generalized solution.



\section{Birkhoff transformation}\label{sec:Birkhoff}

 We now normalize the positions of the two primaries by setting $E=-1$ and $M=+1 $ and denote  by 
\begin{equation}\label{eq:Birkhoffmap}
B \colon \C\setminus \{ 0\} \to \C, \quad B(z) = \frac{1}{2}\left( z + \frac{1}{z}\right)
\end{equation}
the Birkhoff regularization map.


Let $z \colon S^1 \to \C$ be a continuous map, and define its collision set by
\[
\mathfrak{C}_z:=\{ \tau \in S^1 \mid z(\tau) = \pm 1 \}.
\]
Assume that this set is finite and that the following integral is well-defined:
\begin{equation*}\label{eq:hatz}
\hat{z}:= \int_0^1  w(z(\tau)) d\tau <\infty,
\end{equation*}
where
\begin{equation}\label{eq:w}
w(z) = \frac{\lvert z - 1\rvert^2\lvert z +1\rvert^2} {4 \rvert z \rvert^2}  
\end{equation}
Note that $w$ is invariant under the inversion $z \mapsto \frac{1}{z}.$

 We introduce a reparametrization map $t_z \colon S^1 \to S^1,$ defined by 
\begin{equation}\label{eq:timerepart_z}
t_z(\tau):= \frac{1}{\hat z} \int_0^\tau  w(z(s))ds  .
\end{equation}
 Its derivative is given by
\begin{equation*}\label{eq:derivativeoftz}
t_z'(\tau) = \frac{w(z(\tau))}{\hat{z}}
\end{equation*}
so that the set of   critical point of $t_z$ coincides with the collision set:
\[
{\rm crit} (t_z)=\mathfrak{C}_z.
\]
Finiteness of   $\mathfrak{C}_z$  ensures that   the map $t_z$ is a strictly increasing continuous function with finitely many critical points, and thus defines  a  homeomorphism of $S^1.$ In particular, there exists a continuous inverse
\[
\tau_z:=t_z^{-1} \colon S^1 \to S^1
\]
which is smoothly differentiable on the complement $S^1 \setminus t_z(\mathfrak{C}_z),$ with derivative
\begin{equation*}\label{eq:derivativeoftau}
    \dot{\tau}_z(t) = \frac{\hat{z}}{w(z_{\tau}(\tau_z(t)))}
\end{equation*}
Using this, we define a continuous map  
\begin{equation}\label{eq:qz}
q_z \colon S^1 \to \C, \quad q_z(t):= B(z(\tau_z(t))) = \frac{1}{2}\left( z(\tau_z(t)) + \frac{1}{z(\tau_z(t))} \right),
\end{equation}
where $B$ denotes the Birkhoff map, see \eqref{eq:Birkhoffmap}. 
Since   $B$ is a $2$-to-$1$ covering branched at $\pm1,$ the collision set of $q$
\[
   \mathfrak{C}_q  := \{ t \in S^1 \mid q(t) = \pm 1 \} 
   \]
     satisfies the relation 
\[
\mathfrak{C}_z = \tau_z (\mathfrak{C}_q) \quad {\rm or\;\;equivalently\;\;}\;\mathfrak{C}_q = t_z(\mathfrak{C}_z).
\]
In particular, the  collision sets $\mathfrak{C}_z$ and $\mathfrak{C}_q$ are in one-to-one correspondence under the homeomorphisms $\tau_z$  and   $t_z$.


\section{Birkhoff type BOV-regularization}
We study Birkhoff type BOV-regularization of the  time-dependent planar  two-center Stark-Zeeman systems, described by the Newtonian equation
 \begin{equation}\label{eq:twistedODEforSZ}
\ddot{q}   = -\mathfrak{B}(q) i \dot{q} - \frac{(1-\mu) (q+1)}{\lvert q+1\rvert^3 } - \frac{\mu(q-1)}{\lvert q - 1 \rvert^3} - \nabla E_t(q).
\end{equation}
As explained in the introduction, periodic solutions that avoid collisions with $\pm1 $ are critical points of the action functional $\mathscr{A}\colon \mathfrak{L}\mathfrak{D} \to \R$ defined by
 \[  \mathscr{A}(q)  =   \frac{1}{2}\|\dot{q} \|^2  - \int_{S^1} q^*A      + \int_0^1 \frac{ 1-\mu }{\lvert q(t)+1\rvert}  dt + \int_0^1  \frac{\mu}{\lvert q(t) - 1\rvert} dt   - \int_0^1  E_t(q(t))  dt
 \]

Let $\mathfrak{Z}_0:=B^{-1}(\mathfrak{D}_0)$ denote the open subset of $\C$ containing  $ \pm1 ,$   which is the preimage of $\mathfrak{D}_0$ under the Birkhoff map $B.$ Note that $\mathfrak{Z} = \mathfrak{Z}_0\setminus \{ \pm 1 \} = B^{-1}(\mathfrak{D}).$
Since every $z \in \mathfrak{L}\mathfrak{Z}$ does neither hit $+1$  nor $-1,$   the map $t_z$ defined in \eqref{eq:timerepart_z} (and hence  its inverse $\tau_z$) is a diffeomorphism of the circle. 
We define
\begin{equation}\label{eq:SIGMA}
\Sigma \colon \mathfrak{L}\mathfrak{Z} \to \mathfrak{L}\mathfrak{D}, \quad \Sigma(z):=   q_z,
\end{equation}
where $q=q_z$ is given as in \eqref{eq:qz}.  
We compute
\begin{equation}\label{eq:dotqandz}
\dot{q}(t) = \frac{1}{2}\left( 1 - \frac{1}{z(\tau)^2}\right)  \frac{\hat{z}}{w(z(\tau ))} z'(\tau)
\end{equation}
and
\begin{align*}
\ddot{q}(t) &=  \Bigg[  \frac{z''(\tau)}{2} \frac{z(\tau)^2-1}{z(\tau)^2}  +   \frac{z'(\tau)^2}{z(\tau)^3}   - \frac{ z'(\tau)  }{2}\frac{z(\tau)^2-1}{z(\tau)^2}      \frac{\frac{dw}{dz} z'(\tau) + \frac{dw}{d\bar{z}}\bar{z}'(\tau) }{w(z(\tau))}      \Bigg]\frac{\hat{z}^2}{w(z(\tau))^2} \\
&=   \Bigg[  \frac{z(\tau)^2-1}{z(\tau)^2}  z''(\tau)  +   \frac{1-z(\tau)^2}{ z(\tau)^3} z'(\tau)^2 - \frac{ (z(\tau)^2-1)(\bar{z}(\tau)^2+1)}{ z(\tau)\lvert z(\tau)\rvert^2 (\bar{z}(\tau)^2-1) }\lvert z'(\tau)\rvert^2\Bigg]\frac{\hat{z}^2}{2w(z(\tau))^2}\end{align*}
where $\tau = \tau_z(t).$

We  now pull back the functional $\mathscr{A}$ via the map $\Sigma$: 
\begin{align*}
 \label{eq:regularizedfunctionalSZ}   \Sigma^*\mathscr{A}(z) &= \frac{\hat{z}}{2}\int_0^1 \frac{ \lvert z'(\tau)\rvert^2}{\lvert z(\tau)\rvert^2}d\tau - \int_{S^1} z^*{B}^*A +\frac{1-\mu}{2\hat{z}} \int_0^1 \frac{\lvert z(\tau)-1\rvert^2}{\lvert z(\tau)\rvert}d\tau \\
\nonumber    &\;\;\;\;+\frac{\mu}{2\hat{z}} \int_0^1\frac{\lvert z(\tau)+1\rvert^2}{\lvert z(\tau)\rvert}d\tau - \frac{1}{\hat{z}} \int_0^1 E_{t_{z}(\tau)} (B(z(\tau))) \frac{\lvert z(\tau)-1\rvert^2\lvert z(\tau)+1\rvert^2}{4\lvert z(\tau)\rvert^2}d\tau
\end{align*}
This expression naturally extends to a  functional $\mathscr{B} \colon \mathfrak{L}^*\mathfrak{Z}_0  \to \R,$ defined on the extended domain 
\begin{equation*}\label{eq:LstarZ0}
\mathfrak{L}^*\mathfrak{Z}_0:=\{ z \in \mathfrak{L}\mathfrak{Z}_0 \mid \hat{z}\neq 0\}.
\end{equation*}
via the same formula.
To streamline notation, we introduce the functionals
\begin{align*}
\nonumber \mathcal{F}(z) &= \hat{z} = \int_0^1   w(z(\tau))  d\tau, \quad  \quad \quad \mathcal{G}(z) = \frac{1}{2} \int_0^1 \frac{\lvert z'(\tau)\rvert^2}{\lvert z(\tau)\rvert^2}d\tau \\
\label{eq:functionalsEuler} \mathcal{H}_1(z)  &=  \frac{1}{2} \int_0^1   \frac{\lvert z(\tau)-1\rvert^2}{\lvert z(\tau) \rvert}  d\tau ,       \quad  \quad  \quad \mathcal{H}_2(z) =\frac{1}{2} \int_0^1   \frac{\lvert z(\tau)+1\rvert^2}{\lvert z(\tau) \rvert}  d\tau \\
\nonumber \mathcal{M}(z)&= \int_{S^1}z^*B^*A, \quad     \quad \mathcal{E}(z)= \frac{1}{\hat{z}}  \int_0^1 E_{t_{z}(\tau)} (B(z(\tau))) \frac{\lvert z(\tau)-1\rvert^2\lvert z(\tau)+1\rvert^2}{4\lvert z(\tau)\rvert^2}d\tau
\end{align*}
so that
\[ 
\mathscr{B}(z) =   \mathcal{F}(z)\mathcal{G}(z) +\frac{(1-\mu)\mathcal{H}_1(z)+\mu \mathcal{H}_2(z)}{ \mathcal{F}(z)}- \mathcal{M}(z) - \mathcal{E}(z).
\]

We compute the  differentials of the above functionals individually. Choose $z \in \mathfrak{L}^*\mathfrak{Z}_0$ and $\xi \in T_z \mathfrak{L}^*\mathfrak{Z}_0= \mathfrak{L}\C.$
We compute that
\begin{align}
\nonumber d\mathcal{F}(z)\xi &= \int_0^1   \left< \frac{(z-1)\lvert z+1\rvert^2  +(z+1)\lvert z-1\rvert^2}{{2\lvert z \rvert^2} } - \frac{z\lvert z-1\rvert^2 \lvert z+1\rvert^2 }{2\lvert z\rvert^4}  , \xi \right> d\tau \\
\label{eq:differentaildF} &= \int_0^1   \left< \frac{ z(z^2-1)(\bar{z}^2+1) }{2\lvert z\rvert^4}  , \xi \right> d\tau
\end{align}
and further that
\begin{align*}
    d\mathcal{G}(z)\xi &= \int_0^1 \left< \frac{z'}{\lvert z\rvert^2}, \xi'\right> d\tau - \int_0^1 \left< \frac{\lvert z'\rvert^2}{\lvert z\rvert^4}z, \xi \right>d\tau \\
    &= - \int_0^1 \left< \frac{d}{d\tau} \frac{z'}{\lvert z\rvert^2}, \xi \right> d\tau - \int_0^1 \left< \frac{\lvert z'\rvert^2}{\lvert z\rvert^4}z, \xi \right>d\tau \\
        &=  \int_0^1 \left< -  \frac{z''}{\lvert z\rvert^2} + \frac{\bar{z}(z')^2}{\lvert z \rvert^4}, \xi \right> d\tau 
\end{align*}
Moreover, we compute
\begin{align*}
    d\mathcal{H}_1(z)\xi &= \int_0^1 \left< \frac{z-1}{\lvert z\rvert} - \frac{ z\lvert z-1\rvert^2}{2\lvert z \rvert^3} , \xi \right> d\tau=\int_0^1 \left<   \frac{z(z-1)(\bar{z}+1)}{2\lvert z \rvert^3} , \xi \right> d\tau\\
        d\mathcal{H}_2(z)\xi &= \int_0^1 \left< \frac{z+1}{\lvert z\rvert} - \frac{ z\lvert z+1\rvert^2}{2\lvert z \rvert^3} , \xi \right> d\tau =\int_0^1 \left<   \frac{z(z+1)(\bar{z}-1)}{2\lvert z \rvert^3} , \xi \right> d\tau
\end{align*}
 Since the magnetic two-form $\sigma_{\mathfrak{B}}$ pulls back under the Birkhoff map $B$ to 
 \[
 B^* \sigma_{\mathfrak{B}} (z) =  \frac{ \lvert z-1\rvert^2 \lvert z+1\rvert^2}{4\lvert z\rvert^4} \mathfrak{B}( B(z))  dx \wedge dy, \quad \quad z=(x,y),
 \]
the differential of the magnetic part $\mathcal{M}(z)$ is computed as:
\begin{align*}
 d\mathfrak{M}(z)\xi &= \int_0^1 \mathcal{L}_{\xi}( B^*A)(z'(\tau))d\tau\\ 
&= \int_{0}^1B^*\sigma_{\mathfrak{B}}(z(\tau))(\xi(\tau), z'(\tau))d\tau                   \\
 &=   \left<    \frac{\lvert z-1\rvert^2 \lvert z+1\rvert^2}{4\lvert z\rvert^4} \mathfrak{B} (B(z )) i z' , \xi\right>
\end{align*}
To compute the differential of the electric part $\mathcal{E},$ define $\mathcal{N}(z) := \hat{z}\mathcal{E}(z).$ Using the identity \eqref{eq:differentaildF} we compute:
\begin{align*}
     d\mathcal{N}(z)\xi &= \int_0^1 \dot{E}_{t_z(\tau)}(B(z(\tau)))dt_z(\xi)(\tau) w(z(\tau)) d\tau \\
     &\;\;\;\;+ \int_0^1 \left< \nabla E_{t_z(\tau)}(B(z(\tau))), dB(z(\tau))\xi(\tau)\right>w(z(\tau))d\tau \\
     &\;\;\;\;+ \int_0^1E_{t_z(\tau)}(B(z(\tau)))\left<   \frac{ z(\tau)(z(\tau)^2-1)(\bar{z}(\tau)^2+1) }{2\lvert z(\tau)\rvert^4}  , \xi(\tau) \right> d\tau\\
&= \int_0^1 \dot{E}_{t_z(\tau)}(B(z(\tau)))dt_z(\xi)(\tau) w(z(\tau)) d\tau \\
     &\;\;\;\;+ \int_0^1 \left< \nabla E_{t_z(\tau)}(B(z(\tau))), dB(z(\tau))\xi(\tau)\right>w(z(\tau))d\tau \\
     &\;\;\;\;+ \int_0^1E_{t_z(\tau)}(B(z(\tau)))\left<   \frac{ z(\tau)(z(\tau)^2-1)(\bar{z}(\tau)^2+1) }{2\lvert z(\tau)\rvert^4}  , \xi(\tau) \right> d\tau\\
&= \frac{1}{\hat{z}}\int_0^1 \dot{E}_{t_z(\tau)}(B(z(\tau)))  \left( \int_0^\tau  \left<   \frac{ z(s)(z(s)^2-1)(\bar{z}(s)^2+1) }{2\lvert z(s)\rvert^4}    , \xi(s) \right>  ds  \right) w(z(\tau)) d\tau \\
&\;\;\;\;- \frac{1}{\hat{z}^2}\left<   \frac{ z (z ^2-1)(\bar{z} ^2+1) }{2\lvert z \rvert^4}   , \xi \right> \int_0^1 \dot{E}_{t_z(\tau)}(B(z(\tau))) \left(   \int_0^\tau w(z(s))ds         \right) w(z(\tau)) d\tau \\
     &\;\;\;\;+ \int_0^1 \left< \nabla E_{t_z(\tau)}(B(z(\tau)))  \overline{dB(z(\tau))},\xi(\tau)\right>w(z(\tau))d\tau \\
     &\;\;\;\;+ \int_0^1E_{t_z(\tau)}(B(z(\tau)))\left<   \frac{ z(\tau)(z(\tau)^2-1)(\bar{z}(\tau)^2+1) }{2\lvert z(\tau)\rvert^4}  , \xi(\tau) \right> d\tau
     \end{align*}
Therefore, if we introduce
\begin{align*}
\mathcal{E}^1(z) &=   \frac{1}{\hat{z}^2} \int_0^1 \dot{E}_{t_z(\tau)}(B(z(\tau))) \left(   \int_0^\tau w(z(s))ds         \right) w(z(\tau)) d\tau   \\
\varepsilon_1(z)(\tau)&=   \frac{1}{\hat{z}} \left( \int_\tau^1 \dot{E}_{t_z(s)}(B(z(s)) w(z(s))ds \right) \frac{ z(\tau)(z(\tau)^2-1)(\bar{z}(\tau)^2+1) }{2\lvert z(\tau)\rvert^4}     \\
\varepsilon_2(z)(\tau)&=  \nabla E_{t_z(\tau)}(B(z(\tau))) \frac{1}{2}\left( 1 - \frac{1}{\bar{z}(\tau)^2}\right) w(z(\tau))    \\
\varepsilon_3(z)(\tau)&= E_{t_z(\tau)}(B(z(\tau)))    \frac{ z(\tau)(z(\tau)^2-1)(\bar{z}(\tau)^2+1) }{2\lvert z(\tau)\rvert^4}   
\end{align*}
then the differential of $\mathcal{E}$ becomes
\begin{align*}
    d\mathcal{E}(z)\xi  &= \frac{1}{\hat{z}}d\mathcal{N}(z)\xi - \frac{\mathcal{E}(z)}{\hat{z}}d\mathcal{F}(z)\xi \\
    &=  - \frac{ \mathcal{E}(z) + \mathcal{E}^1(z)}{\hat{z}}\left<   \frac{ z (z ^2-1)(\bar{z} ^2+1) }{2\lvert z \rvert^4}   , \xi \right> + \frac{1}{\hat{z}}\left< \varepsilon_1(z)+\varepsilon_2(z)+\varepsilon_3(z), \xi\right>.
\end{align*}

We have thus proven:

\begin{proposition}
  A point $z \in \mathfrak{L}^*\mathfrak{Z}_0$ is a critical point of the regularized functional $\mathscr{B}$ if and only if it is a solution of the following second-order delay differential equation
\begin{align}
 \nonumber   z''(\tau) &= \frac{1}{2\hat{z}^2}\frac{ z(\tau)(z(\tau)^2-1)(\bar{z}(\tau)^2+1) }{\lvert z(\tau)\rvert^2}C  + \frac{\bar{z}(\tau)}{\lvert z(\tau)\rvert^2}z'(\tau)^2\\
 \label{eq:delayequationgeneral}  &\;\;\;\;+ \frac{1}{2\hat{z}^2}\Bigg[ (1-\mu) \frac{z(z-1)(\bar{z}+1)}{  \lvert z \rvert } + \mu   \frac{z(z+1)(\bar{z}-1)}{ \lvert z \rvert }      \Bigg]  \\
 \nonumber    &\;\;\;\; - \frac{w(z(\tau))}{\hat{z}} \mathfrak{B}(B(z(\tau)))iz'(\tau)   - \frac{\lvert z(\tau)\rvert^2}{\hat{z}^2}( \varepsilon_1(z)(\tau)+\varepsilon_2(z)(\tau)+\varepsilon_3(z)(\tau)),
\end{align}
where the constant $C$ is given by
\begin{equation*}\label{eq:constantKgeneral}
C =    \frac{\hat{z}}{2}   \int_0^1 \frac{\lvert z'(\tau)\rvert^2}{\lvert z(\tau)\rvert^2}d\tau    - \frac{1-\mu}{2\hat{z}} \int_0^1   \frac{\lvert z(\tau)-1\rvert^2}{\lvert z(\tau) \rvert}  d\tau  - \frac{\mu}{2\hat{z}} \int_0^1   \frac{\lvert z(\tau)+1\rvert^2}{\lvert z(\tau) \rvert}  d\tau  + \mathcal{E}(z) +\mathcal{E}^1(z)
\end{equation*}
\end{proposition}

To investigate the correspondence between periodic solutions of the Newtonian equation \eqref{eq:twistedODEforSZ} and the delay equation \eqref{eq:delayequationgeneral}, we first extend the definition of the map $t_z,$ originally defined in \eqref{eq:timerepart_z} for $z \in \mathfrak{L}\mathfrak{Z},$ to those $z \in \mathfrak{L}^*\mathfrak{Z}_0$ which are critical points of the regularized functional $\mathscr{B}.$

Assume that $z\in \mathfrak{L}^*\mathfrak{Z}_0$   is a  solution to the delay equation \eqref{eq:delayequationgeneral}. We define $t_z \colon S^1 \to S^1$ exactly by the same formula as in  \eqref{eq:timerepart_z}:
\[
t_z(\tau):=\frac{1}{\hat{z}}\int_0^{\tau}w(z(s))ds.
\]
As before, the collision set is given by
\[
\mathfrak{C}_z = \{ \tau \in S^1 \mid z(\tau) = \pm 1 \} = {\rm crit}(t_z).
\]

\begin{lemma}\label{lem:finitecollisionset}
    The collision set $\mathfrak{C}_z$ is finite.
\end{lemma}
\begin{proof}
The right-hand side of \eqref{eq:delayequationgeneral} can be written in the form 
\[
z''(\tau) = a(z,\tau)z(\tau) + b(z,\tau)\bar{z}(\tau) +c(z,\tau)z'(\tau) 
\]
for continuous functions  $a,b,$ and $c,$ where $a(z,\tau)=b(z,\tau)=0$ whenever $\tau \in \mathfrak{C}_z.$  Assume that $z'(\tau_*)=0$ for some $\tau_* \in \mathfrak{C}_z.$ Then all terms on the right-hand side vanish at $\tau_*,$ implying $z''(\tau_*)=0.$
By the uniqueness theorem of ODE, this would force $ z\equiv  0,$  contradicting the assumption that  $z \in \mathfrak{L}^*\mathfrak{Z}_0.$ Therefore,  $z'(\tau )\neq0$ for every $\tau  \in \mathfrak{C}_z,$ from which we see that the collision set $\mathfrak{C}_z$ is a discrete subset of $S^1,$ and hence finite. 
This completes the proof.
\end{proof}

The previous lemma tells us that the map $t_z$ is a strictly increasing and continuous function having only finitely many critical points. Hence, $t_z$ is a homeomorphism of $S^1,$ and so is its continuous inverse    $\tau_z := t_z^{-1} \colon S^1 \to S^1.$  We then define 
\[
q = q_z \colon S^1 \to \mathfrak{D}_0
\]
as before. Namely, given a critical point $z \in \mathcal{L}^*\mathfrak{Z}_0$ of $\mathscr{B},$ we define
\begin{equation*}\label{eq:qgeneralcase}
q(t) = B(z(\tau_z(t))),
\end{equation*}
which is continuous on the whole circle and smoothly differentiable on the complement of the collision set $\mathfrak{C}_q = t_z(\mathfrak{C}_z).$

Suppose that $z \in \mathfrak{L}^*\mathfrak{Z}_0$ is a solution of the delay equation \eqref{eq:delayequationgeneral} and let $q\colon S^1 \to \mathfrak{D}_0$ be the corresponding map as above.   If  $\mathfrak{C}_z$ is empty, then the map $q$ is smooth and directly satisfies the Newtonian equation   \eqref{eq:twistedODEforSZ} avoiding collisions with $\pm 1. $

We now consider the case $\#\mathfrak{C}_q  = N >0. $ Our goal is to show that $q$ is a generalized periodic   solution of  \eqref{eq:twistedODEforSZ}.
To this end, we decompose the complement $S^1 \setminus \mathfrak{C}_z$ into $N$ open intervals:
\[
S^1 \setminus \mathfrak{C}_z = \bigcup_{j=1}^N \mathcal{I}_j,
\]
where each component is of the form $\mathcal{I}_j:= (\tau_j^-, \tau_j^+),$ and we adopt the cyclic convention 
\[
\tau_j^+ = \tau_{j+1}^-, \;j= 1,2,\ldots, n-1 \quad \quad \text{ and } \quad \quad \tau_n^+ = \tau_1^-.
\]
Let
\[
t_j^\pm :=t_z(\tau_j^\pm), \;\;\; j = 1,\ldots, N
\]
denote the corresponding zeros of the map $q $ and define
\[
\mathcal{J}_j := t_z(\mathcal{I}_j) = (t_j^-, t_j^+) \subset S^1 \setminus \mathfrak{C}_q. 
\]

We define
\[
\mathcal{C}(q) := \int_0^1  E_t(q(t))  dt
\]
so that
\[
\mathcal{E}(z) = \mathcal{C}(q).
\]
We also compute,  by the change of variable $\tau=\tau_z(t)$, that
\begin{align*}
\mathcal{E}^1(z) &=    \int_0^1 \dot{E}_{t_z(\tau)}(B(z(\tau))) \left(   \int_0^\tau \frac{w(z(s))}{\hat{z}}ds         \right) \frac{w(z(\tau))}{\hat{z}} d\tau   \\
&= \int_0^1 \dot{E}_t (q(t)) \left( \int_0^t ds \right) dt\\
&= \int_0^1 t \dot{E}_t (q(t)) dt\\
&=: \mathcal{C}^1(q).
\end{align*}
Moreover, using  \eqref{eq:dotqandz} we find
\[
\frac{z(\tau)^2-1}{2z(\tau)^2}\frac{\hat{z}}{w(z(\tau))}\mathfrak{B}(B(z(\tau))) i z'(\tau) = \mathfrak{B}(q(t))i \dot{q}(t) .
\]
Using the expressions for $\dot{q}(t)$ and  $\ddot{q}(t)$ derived above    and substituting the delay equation \eqref{eq:delayequationgeneral},    we   find on each interval $\mathcal{J}_j$ that 
 \begin{align*}
\nonumber \ddot{q}(t) &= \frac{4 \bar{z}(\tau) (\bar{z}(\tau)^2+1)}{ ( \bar{z}(\tau)^2-1)^2}C        - \frac{2 \bar{z}(\tau) (\bar{z}(\tau)^2+1)}{( \bar{z}(\tau)^2-1)^2}\lvert \dot{q}(t)\rvert ^2 + \frac{4(1-\mu) \lvert z(\tau)\rvert \bar{z}(\tau)}{  (\bar{z}(\tau)-1)^2\lvert z(\tau)+1\rvert^2  } \\
\nonumber & \;\;\;\; + \frac{4\mu  \lvert z(\tau)\rvert \bar{z}(\tau)}{ (\bar{z}(\tau) +1)^2 \lvert z(\tau)-1\rvert^2 }  - \frac{z(\tau)^2-1}{2z(\tau)^2} \frac{\hat{z}}{w(z(\tau))} \mathfrak{B}(B(z(\tau)))iz'(\tau)\\
\nonumber & \;\;\;\;- \frac{z(\tau)^2-1}{2z(\tau)^2}\frac{\lvert z(\tau)\rvert^2}{w(z(\tau))^2}( \varepsilon_1(z)(\tau)+\varepsilon_2(z)(\tau)+\varepsilon_3(z)(\tau)) \\
\nonumber &= \frac{2\bar{q}(t)}{\bar{q}(t)^2-1}\left(C - \frac{1}{2}\lvert \dot{q}(t)\rvert^2\right)    +  \frac{1-\mu}{(\bar{q}(t)-1)\lvert q(t)+1\rvert}+  \frac{ \mu}{(\bar{q}(t)+ 1)\lvert q(t)-1\rvert}\\
\nonumber &\;\;\;\;- \mathfrak{B}(q(t))i\dot{q}(t)-  \frac{2\bar{q}(t)}{ (\bar{q}(t)^2-1)} \int_t^1 \dot{E}_s (q(s))ds\\
  & \;\;\;\; -  \nabla E_{t}(q(t))  - \frac{2\bar{q}(t)}{(\bar{q}(t)^2-1)}E_t (q(t))
 \end{align*}
where the constant $C$  is
\[
C = \int_0^1 \left( \frac{1}{2}\lvert \dot{q}(t)\rvert^2 - \frac{1-\mu}{\lvert q(t)+1\rvert}-\frac{\mu}{\lvert q(t)-1\rvert} \right)dt +\mathcal{C}(q) +\mathcal{C}^1(q).
\]
This can be rearranged as
\begin{align}\label{eq:fromztoqequationgeneral}
  \ddot{q}(t) +\mathfrak{B}(q(t)) i \dot{q}(t) +\nabla E_t(q(t))= \frac{2\bar{q}(t)}{\bar{q}(t)^2-1}\Phi(t) - \nabla U(q(t)),
\end{align}
where  
\[
U(q) = -\frac{1-\mu}{\lvert q+1\rvert}-\frac{\mu}{\lvert q-1\rvert}
\]
and
\begin{equation}\label{eq:Phigeneral}
  \Phi(t):=  C - \frac{1}{2}\lvert \dot{q}(t)\rvert^2 -U(q(t))-  \int_t^1 \dot{E}_s (q(s))ds- E_t (q(t)).
\end{equation}

\begin{lemma}\label{lem:generalfirstintegral}
The   expression 
    \[
\frac{1}{2}\lvert \dot{q}(t)\rvert^2+ U(q(t))+  \int_t^1 \dot{E}_s (q(s))ds+ E_t (q(t)) 
\]
    is conserved along solutions of \eqref{eq:fromztoqequationgeneral}, that is $\Phi \equiv 0.$
\end{lemma}
 \begin{proof}  
We first notice that, by definition of  $C, $ the mean value of $\Phi$ over $S^1$ is zero:
\begin{equation}\label{eq:meanvalueofPhi}
    \int_0^1\Phi(t)dt =0.
\end{equation}
Differentiating  \eqref{eq:Phigeneral} with respect to $t$, we   obtain  
\begin{align*}
\dot{\Phi}(t) &= - \left< \dot{q}(t), \ddot{q}(t)\right> - \left< \dot{q}(t), \nabla U(q(t))\right> +\dot{E}_t(q(t))-\dot{E}_t(q(t)) - \left< \nabla E_t(q(t)), \dot{q}(t)\right>\\
&= \left<  \dot{q}(t) , \mathfrak{B}(q(t)) i \dot{q}(t) +\nabla E_t(q(t))    \right>- \left< \dot{q}(t) , \frac{2\bar{q}(t)}{\bar{q}(t)^2-1}\Phi(t)-\nabla U(q(t))  \right>  \\
&\;\;\;\;\;- \left< \dot{q}(t), \nabla U(q(t))\right> - \left< \nabla E_t(q(t)), \dot{q}(t)\right>\\
&= g(t) \Phi(t), 
\end{align*}
where we define
\[
g(t):=- \left< \dot{q}(t) , \frac{2\bar{q}(t)}{\bar{q}(t)^2-1}\right>. 
\]
Thus, $\Phi(t)$ satisfies the ODE
\begin{equation*}\label{eq:ODEphit}
\dot{\Phi}(t) = g(t)\Phi(t). 
\end{equation*}
We now write this ODE  in terms of $z$ and $\tau=\tau_z(t)$:
\begin{equation*}\label{eq:ODEphitau}
{\Phi}'(\tau) = f(\tau)\Phi(\tau),
\end{equation*}
where  
\begin{align*}
    f(\tau)  := & \; g(t_z(\tau)) \frac{dt}{d\tau}\\
    =& \;   - \left< \dot{q}(t) , \frac{2\bar{q}(t)}{\bar{q}(t)^2-1} \right>   \frac{dt}{d\tau}    \\
    =& \;   - \left<  \frac{1}{2}\left( 1 - \frac{1}{z(\tau)^2}\right)  \frac{\hat{z}}{w(z(\tau ))} z'(\tau) , \frac{4 \bar{z}(\tau) (\bar{z}(\tau)^2+1)}{ ( \bar{z}(\tau)^2-1)^2}\right> \frac{w(z(\tau))}{\hat{z}}     \\
    =& \;   -{\rm Re}\Bigg[    \frac{2(z(\tau)^2+1)}{z(\tau)(z(\tau)^2-1)} z'(\tau)     \Bigg]    \\
    =& \;   {\rm Re}\Bigg[ \left(   \frac{2}{z(\tau)}-\frac{2}{z(\tau)-1}-\frac{2}{z(\tau)+1}     \right)z'(\tau)   \Bigg ]   
\end{align*}
Choose any $\tau_0 \in \mathcal{I}_j$  and solve  the above ODE in $\mathcal{I}_j$: 
 \begin{align*}
     \Phi(\tau)&= \Phi(\tau_0) \exp \left( \int_{\tau_0}^{\tau} f(s) ds\right) \\
     &= \Phi(\tau_0) \exp \left({\rm Re} \int_{\tau_0}^{\tau}   \left(   \frac{2}{z(s)}-\frac{2}{z(s)-1}-\frac{2}{z(s)+1}     \right)z'(s)   ds \right)\\
     &= \Phi(\tau_0) \exp \left( \log   \frac{\lvert z(\tau) \rvert^2 }{ \lvert z(\tau)^2-1 \rvert^2     } + \text{constant} \right) \\
     &=   \Phi(\tau_0)   \frac{\lvert z(\tau) (z(\tau_0)^2-1) \rvert^2 }{\lvert z(\tau_0)(z(\tau)^2-1) \rvert^2} , \quad \tau \in \mathcal{I}_j
 \end{align*}
with a suitable choice of branch for the logarithm. This shows that the function
\begin{align*}
\Psi(\tau):=& \; \Phi(\tau) \frac{\lvert z(\tau)^2-1\rvert^2}{\lvert z(\tau)\rvert^2} \\
=& \; \frac{\lvert z(\tau)^2-1\rvert^2}{\lvert z(\tau)\rvert^2} \Bigg[ C- \int_t^1\dot{E}_s( B(z(\tau_z(s)))    )ds - E_t( B(z(\tau_z(t))))\Bigg] \\
& \;\;- 2 \hat{z}^2 \lvert z'(\tau )\rvert^2 \lvert z(\tau)\rvert^2  +\frac{2(1-\mu)\lvert z(\tau) -1\rvert^2}{ \lvert z(\tau) \rvert } +\frac{2\mu \lvert z(\tau) +1\rvert^2}{\lvert z(\tau) \rvert } 
\end{align*}
is constant on each component $\mathcal{I}_j.$ Since $z \in \mathfrak{L}^*\mathfrak{Z}_0,$ the function $\Psi$ is continuous on the whole circle, and hence there exists a global constant $K \in \R$ such that
\[
\Psi(\tau) = K, \quad \quad \forall \tau \in S^1,
\]
so that
\[
\Phi(\tau) = \frac{\lvert z(\tau)\rvert^2}{\lvert z(\tau)^2-1\rvert^2}K, \quad \quad \tau \in S^1 \setminus \mathfrak{C}_z.
\]
This in particular implies that $\Phi$ does not change its sign and therefore has to vanish because of   \eqref{eq:meanvalueofPhi}.  This completes the proof. 
\end{proof}

 Lemma \ref{lem:generalfirstintegral} shows that   the second-order ODE
\eqref{eq:fromztoqequationgeneral} recovers the Newtonian equation \eqref{eq:twistedODEforq} of the time-dependent two-center Stark-Zeeman system, implying that critical points   of the regularized functional $\mathscr{B}$  correspond to  generalized periodic solutions of the time-dependent two-center Stark-Zeeman system.


 \medskip

In the remainder of this note, we study more closely the relationship between   the critical points of   $\mathscr{B}$ and generalized periodic solutions of \eqref{eq:twistedODEforq}, with particular attention to the underlying symmetry.

Since the function $w(z)$ defined in \eqref{eq:w}  is invariant under the inversion $\C \to \C, \; z \mapsto 1/z,$ it follows immediately that  the regularized functional $\mathscr{B} \colon \mathfrak{L}^*\mathfrak{Z}_0 \to \R$   is invariant under the   involution
\[
I \colon \mathfrak{L}^*\mathfrak{Z}_0 \to \mathfrak{L}^*\mathfrak{Z}_0, \quad I(z)=\frac{1}{z},
\]
that is,  
\begin{equation}\label{eq:symmetry}
\mathscr{B} \circ I = \mathscr{B}.
\end{equation}
Therefore, for all $z \in \mathfrak{L}^*\mathfrak{Z}_0$ and $ \xi \in T_z \mathfrak{L}^*\mathfrak{Z}_0,$ we have
\[
d\mathscr{B}(I(z))dI(z)\xi = d\mathscr{B}(z)\xi, \quad \forall \xi \in T_z \mathfrak{L}^* \mathfrak{Z}_0, \; \forall z \in \mathfrak{L}^* \mathfrak{Z}_0,
\]
which shows that  if $z$ is a critical point of $\mathscr{B},$ then so is $I(z).$ See Figure \ref{fig:inversion_symmetry}. 
 \begin{figure}[!ht]
    \includegraphics[width=0.4\textwidth]{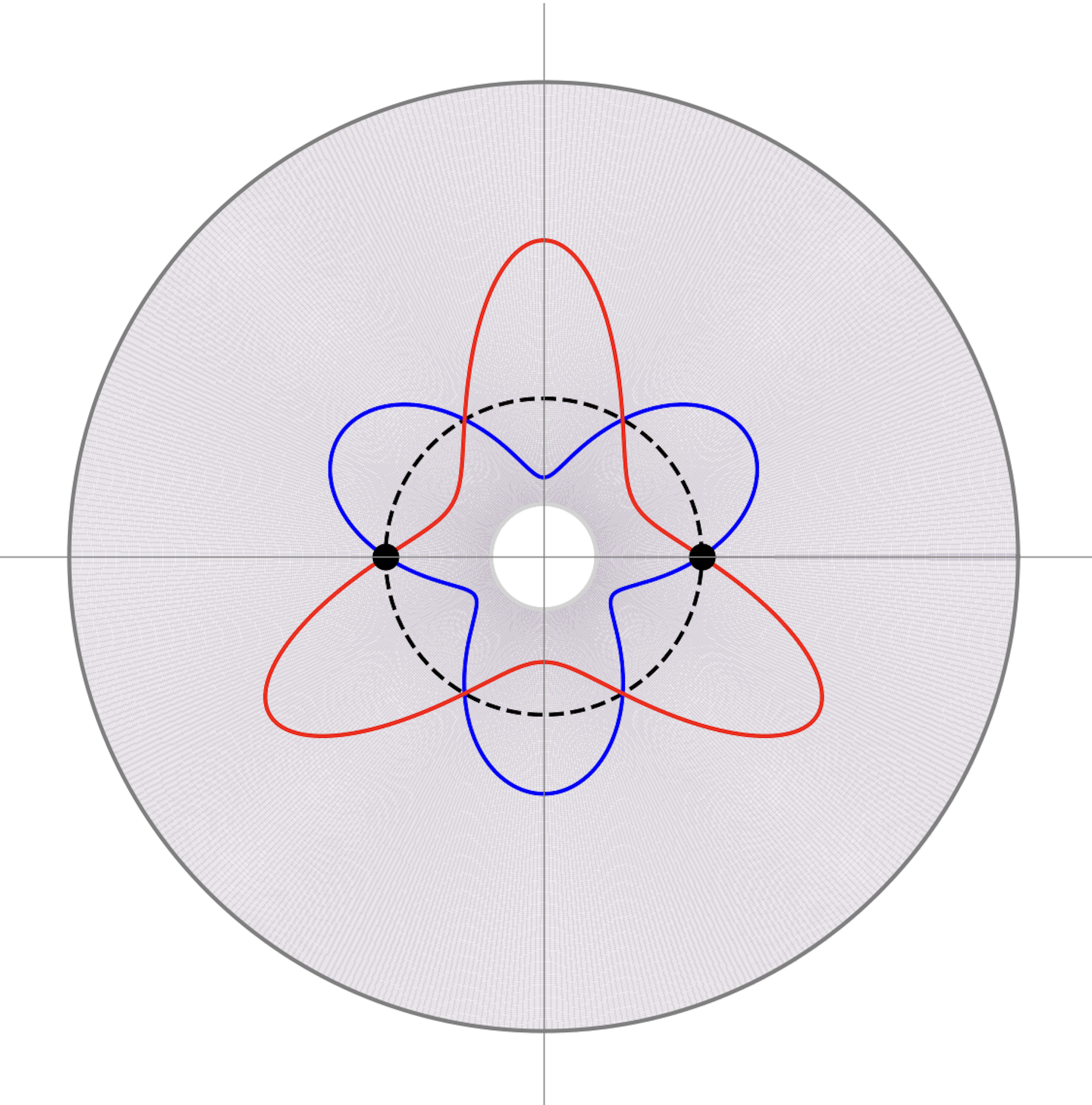}
    \caption{The red and blue periodic orbits are symmetric under the involution $I.$
    The dashed curve indicates the unit circle}
    \label{fig:inversion_symmetry}
\end{figure}
Consequently,  both critical points $z$ and $I(z)$  correspond to the same generalized periodic solution $q_z$ of    \eqref{eq:twistedODEforq}, namely, 
\[
\Sigma(z) = \Sigma(I(z)) = q_z,
\]
where  $\Sigma$ is the map defined   in \eqref{eq:SIGMA}.

Recall that if the total winding number $w(K)$ of a loop $K$ in $\C \setminus \{ E, M\}$--the sum of the winding numbers around the two primaries--is even, then its preimage $B^{-1}(K)$ under the Birkhoff map has two connected components; if odd, then $B^{-1}(K)$ is connected.  See \cite[Proposition 4.2]{CFZ23SZ}. 

For each   $n \in \Z,$ let    $\mathfrak{L}_n\mathfrak{D}  \subset \mathfrak{L}\mathfrak{D}  $   be the subset of loops  with   total winding number   equal to  $n.$ This yields the   decomposition
\[
\mathfrak{L}\mathfrak{D} = \bigsqcup_{n \in \Z} \mathfrak{L}_n\mathfrak{D}
\]
In a similar manner, we define the decomposition  
\[
\mathfrak{L}\mathfrak{Z} = \bigsqcup_{n \in \Z} \mathfrak{L}_n\mathfrak{Z}
\]
Note that $z$ and $I(z)$ share  the same total winding number. 
Since the Birkhoff map $B$ is a branched double cover with two branch points at $E$ and $M,$ it doubles the total winding number. More precisely, if $z \in \mathfrak{L}_n\mathfrak{Z} ,$ then and $q := B(z) \in \mathfrak{L}_{2n}\mathfrak{D}.$ Therefore, for each $n \in \Z$ the map $\Sigma$ induces a diffeomorphism
\[
\Sigma_n \colon \mathfrak{L}_n \mathfrak{Z} / I \to \mathfrak{L}_{2n}\mathfrak{D}
\]

To account for loops with odd winding number, note that if  $q \in \mathfrak{L}_{2n+1}\mathfrak{D}$, then $B^{-1}(q)$ is connected. Following \cite[Section 7]{Fra25SZ}, we introduce   the twisted loop space
\[
\mathfrak{L}_I\mathfrak{Z} := \{ z \in C^\infty(\R, \mathfrak{Z}) \mid z(t+1) = I \circ z(t), \; t \in \R \}.
\]
Note that every element of $\mathfrak{L}_I\mathfrak{Z}$ is not a loop in the usual sense, but   satisfies  a twisted periodicity: as time increases by $1,$ a point is mapped to its image  under the involution $I$.   Given a twisted loop $z \in \mathfrak{L}_I\mathfrak{Z},$ we recover a genuine loop $\tilde z \in \mathfrak{L}\mathfrak{Z}$ by
\[
 \tilde{z}(t):=z(2t), \quad t \in \R.
 \]
 Indeed, using the twisted periodicity we find
 \[
 \tilde{z}(t+1) = z(2t+2) = I \circ z ( 2t+1) = z(2t) = \tilde{z}(t) .
 \]
Thus,  twisted loops correspond to loops with odd winding number. We define the  decomposition 
\[
\mathfrak{L}_I\mathfrak{Z} = \bigsqcup_{n \in \Z}\mathfrak{L}_{I, n+\frac{1}{2}}\mathfrak{Z},
\]
where
\[
\mathfrak{L}_{I, n+\frac{1}{2}} \mathfrak{Z} := \{ z \in \mathfrak{L}_I\mathfrak{Z} \mid \tilde{z} \in \mathfrak{L}_{2n+1}{\mathfrak{Z}}\}
\]
We now define the map 
\[
\Sigma_I \colon \mathfrak{L}_I\mathfrak{Z} \to \mathfrak{L}\mathfrak{D}, \quad \Sigma_I(z):= q_z , 
\]
where $q_z$ is defined as in \eqref{eq:qz}. Although $z$ is not a loop, but a twisted loop, the map $q_z$ is an actual loop:  
\begin{align*}
q_z(t+1) &= \frac{1}{2}\left( z(\tau_z(t+1)) + \frac{1}{z(\tau_z(t+1))}\right)\\
&= \frac{1}{2}\left( z(\tau_z(t)+1 ) + \frac{1}{z(\tau_z(t)+1 )}\right)\\
&= \frac{1}{2}\left( I \circ z(\tau_z(t)) + \frac{1}{ I \circ z(\tau_z(t))}\right)\\
&=\frac{1}{2}\left(   \frac{1}{z(\tau_z(t+1))} + z(\tau_z(t+1)) \right)\\
&= q_z(t)
\end{align*}
Hence, for each $n \in \Z,$ we have  a diffeomorphism
\[
\Sigma_{ I,n } \colon \mathfrak{L}_{{2n}+1} \mathfrak{Z} /I\to \mathfrak{L}_{2n+1}\mathfrak{D}.
\]
 
 Combining the even and odd cases,   we define the extended map
\[
 \Sigma_0:=\Sigma \oplus \Sigma_I \colon \mathfrak{L}\mathfrak{Z} \cup \mathfrak{L}_I \mathfrak{Z} \to \mathfrak{L}\mathfrak{D}
 \]
and   consider the  pullback  functional 
 \[
\Sigma_0^* \mathscr{A} \colon \mathfrak{L}\mathfrak{Z} \cup \mathfrak{L}_I \mathfrak{Z} \to \R.
 \]
Then we obtain a one-to-one correspondence between the critical points
  \begin{align*}
  {\rm Crit}(\Sigma_0^*\mathscr{A})/I \;\;\; \xleftrightarrow{\hspace{1cm}} \;\;\; {\rm Crit} (\mathscr{A})
  \end{align*}
with
\[
  \text{critical points of winding number $\frac{n}{2}$}  \; \xmapsto{\hspace{0.5cm}} \;  {\text{critical points of winding number $n$}}
\]
To include  generalized solutions,  define the extended twisted loop space
\[
\mathfrak{L}_I^*\mathfrak{Z}_0 :=\{z \in \mathfrak{L}_I \mathfrak{Z}_0 \mid \hat{z}\neq 0 \}
\]
and   extend the regularized functional:
\[
\mathscr{B}_0 \colon \mathfrak{L}^*\mathfrak{Z}_0 \cup \mathfrak{L}_I^*\mathfrak{Z}_0  \to \R.
\]
Thus, as before, we  have a bijective  correspondence 
  \[ 
  {\rm Crit} (\mathscr{B}_0)/I \;  \xleftrightarrow{1:1}\; \{\text{generalized periodic solutions of \eqref{eq:twistedODEforq}}\} 
  \]

We summarize the result of this note in the following.

\begin{theorem}\label{thm:main}
    There is a one-to-one correspondence between critical points of the extended regularized functional $\mathscr{B}_0$ modulo the involution $I$ and generalized periodic solutions of the time-dependent Stark-Zeeman system \eqref{eq:twistedODEforq}.
\end{theorem}

\bibliographystyle{abbrv}
\bibliography{mybibfile}

\end{document}